\documentclass[11 pt]{amsart}
\usepackage[active]{srcltx}
\usepackage{calc,amssymb,amsthm,amsmath, ulem}
\usepackage{alltt}
\usepackage[top=1.6in,right=1.4in,left=1.4in,bottom=1.6in]{geometry}

\RequirePackage[dvipsnames,usenames]{color}  

\DeclareFontFamily{OMS}{rsfs}{\skewchar\font'60}
\DeclareFontShape{OMS}{rsfs}{m}{n}{<-5>rsfs5 <5-7>rsfs7 <7->rsfs10 }{}
\DeclareSymbolFont{rsfs}{OMS}{rsfs}{m}{n}
\DeclareSymbolFontAlphabet{\scr}{rsfs}

\newtheorem{theorem}{Theorem}[section]
\newtheorem{lemma}[theorem]{Lemma}
\newtheorem{proposition}[theorem]{Proposition}
\newtheorem{corollary}[theorem]{Corollary}

\theoremstyle{definition}
\newtheorem{definition}[theorem]{Definition}

\theoremstyle{remark}
\newtheorem{remark}[theorem]{Remark}

\newtheorem{question}[theorem]{Question}
\input xy
\xyoption{all}

\theoremstyle{definition}
\newtheorem{definition-proposition}[theorem]{Definition-Proposition}

\newcommand{\reduced}{non-degenerate}
\newcommand{\ureduced}{non-degenerately generated}
\newcommand{\freegen}{pseudo-$\Q$-Gorenstein}
\newcommand{\homgen}{homogeneously defined}
\newcommand{\homgenbs}{HDBS}

\renewcommand{\O}{\mathcal{O}}
\DeclareMathOperator{\vlocus}{{DegenLocus}}
\DeclareMathOperator{\Spec}{{Spec}}
\newcommand{\tld}{\widetilde }
\newcommand{\bQ}{\mathbb{Q}}
\newcommand{\ba}{\mathfrak{a}}
\newcommand{\mJ}{\mathcal{J}}
\DeclareMathOperator{\Hom}{Hom}

\newcommand{\sT}{\scr{T}}
\newcommand{\sC}{\scr{C}}

\DeclareMathOperator{\Supp}{{Supp}}
\newcommand{\sB}{\scr{B}}
\DeclareMathOperator{\Image}{Image}
\newcommand{\bb}{\mathfrak{b}}
\newcommand{\bm}{\mathfrak{m}}
\newcommand{\Q}{\mathbb {Q}}
\newcommand{\tensor}{\otimes}
\newcommand{\sS}{\scr{S}}

\renewcommand{\th}{\ensuremath{^\text{th}}}

\begin{document}

\title{Test ideals in non-$\bQ$-Gorenstein rings}
\author{Karl Schwede}

\address{Department of Mathematics\\ University of Michigan\\ East Hall
530 Church Street \\ Ann Arbor, Michigan, 48109}
\email{kschwede@umich.edu}


\thanks{The author was partially supported by a National Science Foundation postdoctoral fellowship and by RTG grant number 0502170.}

\begin{abstract}
Suppose that $X = \Spec R$ is an $F$-finite normal variety in characteristic $p > 0$.  In this paper we show that the big test ideal $\tau_b(R) = \tld \tau(R)$ is equal to $\sum_{\Delta} \tau(R; \Delta)$ where the sum is over $\Delta$ such that $K_X + \Delta$ is $\bQ$-Cartier.  This affirmatively answers a question asked by various people, including Blickle, Lazarsfeld, K. Lee and K. Smith.  Furthermore, we have a version of this result in the case that $R$ is not even necessarily normal.
\end{abstract}
\maketitle
\section{Introduction}

Suppose that $X = \Spec R$ is a normal $\bQ$-Gorenstein variety in characteristic zero.  For any ideal sheaf $\ba$ on $X$ and any positive real number $t > 0$, we can define the multiplier ideal $\mJ(X; \ba^t)$ which reflects subtle local properties of both $X$ and elements of $\ba$, see for example \cite[Chapter 9]{LazarsfeldPositivity2}.  Furthermore, if one reduces $X$ and $\ba$ to characteristic $p \gg 0$, then the multiplier ideal $\mJ(X_p; \ba_p^t)$ agrees with the big test ideal $\tau_b(R_p; \ba_p^t) = \tld \tau(R_p, \ba_p^t)$ (the big test ideal is also called the non-finitistic test ideal), see \cite{SmithMultiplierTestIdeals}, \cite{HaraInterpretation}, \cite{HaraYoshidaGeneralizationOfTightClosure}, and \cite{TakagiInterpretationOfMultiplierIdeals}.  However, while at least classically, multiplier ideals need the $\bQ$-Gorenstein hypothesis in order to be defined, big test ideals do not.

In characteristic zero, one way to get around this difficulty is to consider an additional $\bQ$-divisor $\Delta$ on $X$ such that $K_X + \Delta$ is $\bQ$-Cartier.  In that setting, one can define the multiplier ideal $\mJ(X; \Delta, \ba^t)$ of the triple $(X, \Delta, \ba^t)$.  Unfortunately, there is no canonical choice of $\Delta$ (although there can be quite ``good'' choices of $\Delta$, see \cite{DeFernexHacon}).  On the other hand, Hara, Yoshida and Takagi defined the big test ideal $\tau_b(R; \Delta, \ba^t)$ of a triple $(X, \Delta, \ba^t)$ and showed it agreed with the multiplier ideal $\mJ(X; \Delta, \ba^t)$ after reduction to characteristic $p \gg 0$, see \cite{HaraYoshidaGeneralizationOfTightClosure} and \cite{TakagiInterpretationOfMultiplierIdeals}.

Suppose now we work in a fixed characteristic $p > 0$.  It is then very natural to ask how the test ideal $\tau_b(R; \ba^t)$ is related to the test ideals $\tau_b(R; \Delta, \ba^t)$ where $\Delta$ ranges over all divisors such that $K_X + \Delta$ is $\bQ$-Cartier.  It is easy to see that one has containments $\tau_b(R; \Delta, \ba^t) \subseteq \tau_b(R; \ba^t)$.  Furthermore, for several years it has been asked whether in fact one has
\[
\tau_b(R; \ba^t) = \sum_{\Delta} \tau_b(R; \Delta, \ba^t)
\]
where the sum is over $\Delta$ such that $K_X + \Delta$ is $\bQ$-Cartier. The main result of this paper is an affirmative answer to this question.  This result is known in the toric case by \cite[Corollary to Theorem 3]{BlickleMultiplierIdealsOnToric} (also see \cite[Theorem 4.8]{HaraYoshidaGeneralizationOfTightClosure}) and Blickle asked whether it always holds.  Another version of this question is found in \cite[Remark 3.6]{LazarsfeldLeeSmithSyzygiesOnSingular}.  A similar variant was asked by Lazarsfeld during the open problems session at the ``$F$-singularities and $D$-modules'' conference held at the University of Michigan in August 2007.  In fact, we show somewhat more.  We show that one can restrict to only considering $\Delta$ so that $K_X + \Delta$ is $\bQ$-Cartier with index not divisible by $p > 0$.  Therefore our main result is the following:
\vskip 6pt \hskip -12pt
{\bf Corollary \ref{CorMain}. }{\it
Suppose that $(X = \Spec R, \ba^t)$ is a pair and that $R$ is an $F$-finite normal domain.  Then
\[
\tau_b(R; \ba^t) = \sum_{\Delta} \tau_b(R; \Delta, \ba^t)
\]
where the sum is over effective $\bQ$-divisors $\Delta$ such that $K_X + \Delta$ is $\bQ$-Cartier with index not divisible by $p > 0$.
}
\vskip 6pt
\hskip -12pt
In fact, we even have a version of this result that works in the case that $X = \Spec R$ is not necessarily normal, see Theorem \ref{ThmMain}.  One can also obtain the following, which may be of independent interest.
\vskip 6pt \hskip -12pt
{\bf Corollary \ref{CorDontNeedBAT}. }{\it
Suppose that $(X = \Spec R, \ba^t)$ is a pair and that $R$ is an $F$-finite normal domain.  Then there exists finitely many effective $\bQ$-divisors $\Delta_i$ such that $K_X + \Delta_i$ is $\bQ$-Cartier with index not divisible by $p$ and also that
\[
\tau_b(R; \ba^t) = \sum_{\Delta_i} \tau_b(R; \Delta_i).
\]
}
\vskip 6pt
\hskip -12pt
It should also be noted that changing the $t$ in the previous corollary will change the $\Delta_i$ in ways which we do not know how to control.

Recently de Fernex and Hacon have worked out a theory of multiplier ideals $\mJ(X; \ba^t)$ for pairs $(X, \ba^t)$ where $X$ is not necessarily $\bQ$-Gorenstein, see \cite{DeFernexHacon}.  In particular, it is known that $\mJ(X; \ba^t) = \sum_{\Delta} \mJ(X; \Delta, \ba^t)$ where the sum ranges over $\Delta$ such that $K_X + \Delta$ is $\bQ$-Cartier (in fact, they show that there exists a \emph{single} $\Delta$ such that $\mJ(X; \ba^t) = \mJ(X; \Delta, \ba^t)$).  Therefore our result could be viewed as additional justification for their definition.

To prove our main result, we use the fact that the non-zero elements $\phi \in \Hom_R(F^e_*R, R)$ induce $\bQ$-divisors $\Delta$ such that $(K_X + \Delta)$ is $\bQ$-Cartier with index not divisible by $p$, see \cite{SchwedeFAdjunction} for details.  Another key point in the proof is the construction of an element $c \in R^{\circ}$ that is simultaneously a big sharp test element for a (specially chosen) infinite collection of triples $(R, \Delta_{\alpha}, \ba^t)$, where $\alpha$ ranges over some indexing set, see Proposition \ref{PropUniformExistenceOfTestElements}.  In proving our result, we also develop a theory of pairs $(R, \sT)$ where $\sT$ is a graded subalgebra of the non-commutative algebra $\oplus_{e \geq 0} \Hom_R(F^e_* R, R)$ (which is Matlis-dual to the non-commutative algebra $\mathcal{F}(E_R(k))$ as studied by Lyubeznik and Smith, \cite{LyubeznikSmithCommutationOfTestIdealWithLocalization}).  In particular, our pairs $(R, \sT)$ are strict generalizations of triples $(R, \Delta, \ba^t)$.  See Remark \ref{RemCanConstructNewAlgebras}.

We conclude the paper with an additional question related to the work of de Fernex and Hacon, a brief discussion of the differences between the big test ideal and the finitistic (classical) test ideal, and a comparison of the work done in this paper with some of the results of \cite{SchwedeSmithLogFanoVsGloballyFRegular}.
\vskip 12pt
\hskip -12pt{\it Acknowledgments: }  The author would like to thank Manuel Blickle, Tommaso de Fernex, Mel Hochster and Shunsuke Takagi for several valuable discussions and encouragement.  In particular, the author would also like to thank Manuel Blickle for several excellent suggestions for terminology.  The author would also like to thank Jeremy Berquist, Tommaso de Fernex and Wenliang Zhang for comments on a previous draft of this paper.  Finally, the author would like to thank the referee for numerous useful comments, suggestions, and corrections.

\section{Preliminaries and notation}

In this section we recall the basic definitions and notations we will study in this paper.  Throughout this paper, all rings will be assumed to be Noetherian.  Unless otherwise specified, all rings considered will be assumed to contain a field of characteristic $p > 0$.  Recall that there is the Frobenius action on any such ring $R$.  This is the ring homomorphism $F : R \rightarrow R$ that sends elements of $R$ to their $p$\th{} powers.  We can also iterate the Frobenius morphism $e$-times, $\xymatrix{R \ar[r]^F & R \ar[r]^F & \ldots \ar[r]^F & R}$, and we denote the composition by $F^e : R \rightarrow R$.

We can then view $R$ as an $R$-module via the action of the $e$-iterated Frobenius.  We will use the notation $F^e_* R$ to denote this module (which we will still sometimes view as a ring in its own right).  Note that this notation is justified, because if $X = \Spec R$ and we abuse notation and use $F^e : X \rightarrow X$ to denote the map of schemes induced by $F^e$, then $F^e_* \O_X$ is the module corresponding to $F^e_* R$.  Likewise, for any $R$-module $M$, we use $F^e_* M$ to denote $M$ viewed as an $R$-module by the $e$-iterated Frobenius action.  This notation also lets us identify the Frobenius map $F^e$ with an $R$-linear map $R \rightarrow F^e_* R$ which we also denote by $F^e$.

\begin{definition}
We say that $R$ is \emph{$F$-finite} if $F^e_* R$ is finite as an $R$-module.
\end{definition}

\begin{remark}
For example, any ring that is essentially of finite type over a perfect field is $F$-finite, see \cite{FedderFPureRat}.
\end{remark}

\emph{All} rings in this paper will be assumed to be $F$-finite.  In particular, since an $F$-finite ring is excellent, all rings in this paper will be assumed to be excellent, see \cite{KunzOnNoetherianRingsOfCharP}.

Throughout this paper, we will consider the module $\Hom_R(F^e_* R, R)$.  While we will not need this directly, it may be useful for the reader to note that if $X = \Spec R$ is normal and has a dualizing complex, then
\[
\Hom_{\O_X}(F^e_* \O_X, \O_X) \cong F^e_* \O_X((1 - p^e)K_X)
\] as $F^e_* \O_X$-modules; see \cite{MehtaRamanathanFrobeniusSplittingAndCohomologyVanishing} and \cite{HaraWatanabeFRegFPure}.

Finally, we briefly remind the reader of the definition of $\bQ$-divisors.
If $X$ is normal, then a $\bQ$-divisor on $X$ is a formal sum of prime divisors on $X$ with rational coefficients.  A $\bQ$-divisor $D$ on $X$ will be called $\bQ$-Cartier if and only if there exists an integer $n > 0$ such that $n D$ is an integral Cartier divisor (ie, $nD$ is locally trivial in the divisor class group).  Abusing notation, if $X = \Spec R$ is normal, and $D$ is any integral divisor on $X$, then we will use $R(D)$ to denote the global sections of $\O_X(D)$.

\section{An ultimate generalization of pairs}

In this section we introduce a (perhaps ultimate) generalization of a pair in the characteristic $p > 0$ setting.  Our notion encompasses the  triples $(R, \Delta, \ba^t)$, see \cite{HaraYoshidaGeneralizationOfTightClosure} \cite{TakagiInterpretationOfMultiplierIdeals}, and \cite{TakagiPLTAdjoint}, as a special case, see Remark \ref{RemCanConstructNewAlgebras}. It also seems well behaved for any $F$-finite ring (we however, restrict to the reduced case).

Suppose that $R$ is an $F$-finite and reduced ring.  For each $e \geq 0$, we can consider the module $\Hom_R(F^e_* R, R)$.  We take the direct sum of these modules together to form a non-commutative graded ring:
\[
\sC(R) = \oplus_{e \geq 0} \Hom_R(F^e_* R, R)
\]
where the multiplication for homogeneous elements
\[
\phi \in \Hom_R(F^d_* R, R) = \sC(R)_d\text{ and }\psi \in \Hom_R(F^e_* R, R) = \sC(R)_e
\]
is defined as follows:
\[
\phi \cdot \psi = \phi \circ (F^d_* \psi) \in \Hom_R(F^{d+e}_* R, R).
\]

\begin{definition}
Given a $\phi \in \Hom_R(F^e_* R, R)$, we define the \emph{degeneracy locus of $\phi$} (denoted $\vlocus(\phi)$) to be the subset of $\Spec R$ made up of primes $Q \in \Spec R$ such that the image of $\phi$ in
\[
\Hom_R(F^e_* R, R)_Q \cong \Hom_{R_Q}(F^e_* R_Q, R_Q)
\]
does \emph{not} generate $\Hom_{R_Q}(F^e_* R_Q, R_Q)$ as an $F^e_* R_Q$-module.
\end{definition}

\begin{remark}
If $R$ is normal and $\phi$ corresponds to a $\bQ$-divisor $\Delta$ as in \cite{SchwedeFAdjunction}, then $\vlocus(\phi) = \Supp(\Delta)$.
\end{remark}

We will now remark that the degeneracy locus is always closed subset of $\Spec R$.

\begin{lemma}
Given $\phi \in \Hom_R(F^e_* R, R)$, then $\vlocus(\phi)$ is a closed subset of $X = \Spec R$.
\end{lemma}
\begin{proof}
It is obvious since $\Hom_R(F^e_* R, R)$ is a finitely generated $F^e_* R$-module.
\end{proof}

\begin{definition}
A \emph{pair $(R, \sT)$} is the combined information of a reduced $F$-finite ring $R$, and a graded subalgebra $\sT \subseteq \sC(R)$ such that $\sT_0 = \sC(R)_0 = \Hom_R(R, R) \cong R$.  A \emph{triple $(R, \sT, \ba^t)$} is the combined information of a pair $(R, \sT)$, an ideal $\ba \subseteq R$ and a positive real number $t > 0$.
\end{definition}

\begin{remark}
Using the assumption that $\sT_0 \cong R$, we immediately see that for each $i$, $\sT_i$ has the structure of an $F^i_* R$-module (note that $F^i_* R$ is also a ring in its own right).
\end{remark}

\begin{remark}
When we consider a pair $(R, \ba^t)$ we implicitly are referring to the triple $(R, \sC(R), \ba^t)$.  Furthermore, while it is true that for any triple, $(R, \sT, \ba^t)$, there is a way to ``absorb'' $\ba^t$ into $\sT$ (and create a smaller subalgebra), see Remark \ref{RemCanConstructNewAlgebras}, we will find it convenient to separate the terms at times.
\end{remark}

\begin{remark}
\label{RemarkNaturalMapsToR}
 One interesting and sometimes useful property of the algebra $\sC(R)$ is the fact that there are natural maps
\[
\rho_{i,j} : \sC(R)_i = \Hom_R(F^i_* R, R) \rightarrow \Hom_R(F^j_* R, R) = \sC(R)_j
 \]
 for $i > j$, induced by the inclusion $F^j_* R \subseteq F^i_* R$.  Our subalgebras are \emph{NOT} assumed to satisfy any analog of this property.  However, we still have natural maps $\sT_i \rightarrow \Hom_R(R, R) \cong R$, for any graded subalgebra $\sT \subseteq \sC(R)$.  These maps are induced by restricting the natural maps $\sC(R)_i \rightarrow \Hom_R(R, R) \cong R$, to $\sT \subseteq \sC(R)$.  Alternately, these are the maps defined by evaluation at $1$.

 Furthermore, suppose that some $F^i_* R$-submodule $\sB \subseteq \sC(R)_i$ is generated (as an $F^i_* R$-module) by elements $\phi_1, \ldots, \phi_n$.  Then notice that $$\Image(\rho_{i,0}|_{\sB} ) = \sum_i \Image(\phi_i).$$  The containment $\subseteq$ follows since each $\phi \in \sB$ can be written as a sum $\sum_j (\phi_j \cdot x_j)$ for some $x_j \in F^i_* R$.  Then
 \[
 \phi(1) \in \Image(\phi) \subseteq \sum_j \Image(\phi_j \cdot x_j) \subseteq \sum_j \Image(\phi_j).
 \]
 The containment $\supseteq$ follows since if $x \in \Image(\phi)$, then there exists some $y \in F^i_* R$ so that $\phi(y) = x$ which implies that $x = (\phi \cdot y) (1)$.
\end{remark}

\begin{definition}  Suppose that $(R, \sT)$ is a pair.  Then, given any multiplicative subset $W$ of $R$, we can construct a new pair $(W^{-1} R, W^{-1} \sT)$ in the obvious way.  A homogeneous element $\phi \in \sT_e$ will be called \emph{\reduced} if its image in $(R_{\eta}, \sT_{\eta})$ is non-zero for every minimal prime $\eta$ of $\Spec R$. We say that a pair $(R, \sT)$ is \emph{\reduced} if it contains a \reduced{} element $\phi \in \sT_e$ for some $e > 0$.  We say that $(R, \sT)$ is \emph{\ureduced{}} if, for every $e > 0$ such that $\sT_e$ is non-zero, we have that $\sT_e$ is generated as an $F^e_* R$-module by \reduced{} elements.

We say that a triple $(R, \sT, \ba^t)$ is \emph{\reduced} if $(R, \sT)$ is \reduced{} and in addition, $\ba \cap R^{\circ} \neq \emptyset$.  We say that $(R, \sT, \ba^t)$ is \emph{\ureduced{}} if $(R, \sT)$ is \ureduced{} and $\ba \cap R^{\circ} \neq \emptyset$.  Note this also implies that the subalgebra $\sT'$ defined by $\sT'_e = \sT_e \cdot \left(F^e_* \ba^{\lceil t(p^e - 1)\rceil} \right)$ is also \ureduced{}.
\end{definition}

\begin{remark}
Because $R$ is a reduced ring, $(R, \sT)$ is \reduced{} if and only if there exists some homogeneous $\phi \in \sT_{> 0}$ such that $X \setminus \vlocus(\phi)$ is open and dense.  Also note that if $(R, \sT, \ba^t)$ is reduced and $\phi \in \sT$ is a reduced homogeneous element, then there exists some $c \in \ba \cap R^{\circ}$  such that $\Supp(\phi) \subseteq V(c) \subseteq \Spec R$.
\end{remark}

\begin{remark}
\label{RemCanConstructNewAlgebras}
From the point of view of $F$-singularities, pairs $(R, \sT)$ (respectively triples $(R, \sT, \ba^t)$) are generalizations of triples $(R, \Delta, \bb_{\bullet})$ (respectively triples $(R, \sT, \bb_{\bullet} \cdot \ba^t)$) as studied by several authors (here $\bb_{\bullet}$ is a graded system of ideals\footnote{Associated with any $\ba^t$ (an ideal $\ba$ formally raised to a real power $t > 0$) there is a graded system defined by $\bb_i = \ba^{\lceil t i \rceil}$.}); see for example \cite{TakagiInterpretationOfMultiplierIdeals}, \cite{HaraYoshidaGeneralizationOfTightClosure}, \cite{HaraACharacteristicPAnalogOfMultiplierIdealsAndApplications}, \cite{BlickleMustataSmithDiscretenessAndRationalityOfFThresholds}, \cite{SchwedeCentersOfFPurity}, \cite[Remark 2.8]{SchwedeFAdjunction}.  To construct the pair $(R, \sT)$ associated to $(R, \Delta, \bb_{\bullet})$, proceed as follows:  Define $\sT_i$ to be
\[
\Image\bigg(\Hom_R(F^i_* R((p^i - 1)\Delta), R) \rightarrow \Hom_R(F^i_* R, R) \bigg) \cdot (F^i_* \bb_{p^i - 1}).
\]
To see that $\oplus_i \sT_i$ is a sub\emph{algebra} of $\sC(R)$, simply make the following two observations.
\begin{itemize}
\item[(1)]  If $b \in \bb_{p^i - 1}$ and $b' \in \bb_{p^j - 1}$, then $b^{p^j} b' \in \bb_{p^j(p^i - 1) + (p^j - 1)} = \bb_{p^{i+j} - 1}$.
\item[(2)]  If $\phi \in \sT_i$ and $\psi \in \sT_j$ then $\phi \circ F^i_* \psi \in \sT_{i+j}$ (this follows since, $p^j \lceil (p^i - 1)\Delta \rceil + \lceil (p^j - 1) \Delta \rceil \geq \lceil (p^{i+j}- 1)\Delta \rceil$).
\end{itemize}
This method also allows one to obtain triples $(R, \sT, \ba^t)$ from other triples $(R, \sT, \bb_{\bullet} \cdot \ba^t)$.  As we will see, changing the triple in this way does not impact the associated test ideals.
\end{remark}

Note that if $R$ is a normal domain and $\ba_i \cap R^{\circ} \neq \emptyset$ for some $i > 0$, then the pair $(R, \sT)$ constructed from the triple $(R, \Delta, \ba_{\bullet})$ is \reduced.

\begin{definition}\cite{HochsterRobertsRingsOfInvariants}, \cite{HaraWatanabeFRegFPure}, \cite{SchwedeSharpTestElements}, \cite{SchwedeCentersOfFPurity}
A pair $(R, \sT)$ is \emph{sharply $F$-pure} if there exists a homogeneous $\phi \in \sT_e$ such that $\phi(F^e_* R) = R$ (that is, $\phi$ is surjective).  A triple $(R, \sT, \ba^t)$ is called \emph{sharply $F$-pure} if there exists a homogeneous $\phi \in  \sT_e \cdot \left(F^e_* \ba^{\lceil t(p^e - 1) \rceil} \right)$ such that $\phi(F^e_* R) = R$.
\end{definition}

\begin{definition} \cite{HochsterHunekeFRegularityTestElementsBaseChange}, \cite{HaraWatanabeFRegFPure}, \cite{TakagiInversion}
We call a triple $(R, \sT, \ba^t)$ \emph{strongly $F$-regular} if for every $d \in R^{\circ}$, there exists an element $\phi \in \sT_e \cdot \left( F^e_* \ba^{\lceil t(p^e - 1) \rceil} \right)$ such that $1 \in \phi(F^e_* (d) ) = \phi(F^e_* (dR) )$ (here $(d) = dR$ is the ideal generated by $d$).
For the definition for pairs $(R, \sT)$, set $\ba = R$.
\end{definition}

Note that the element $\phi$ from the above two definitions is not necessarily an element of the form $c \phi_e$ where $c \in F^e_* \ba^{\lceil t(p^e - 1) \rceil}$ and $\phi_e \in \sT_e$.  In general it is a sum of such elements.

\begin{remark}
It is clear that a strongly $F$-regular triple is sharply $F$-pure.  Furthermore, if $R$ is regular, $\ba = R$ and $\sT = \sC(R)$ then $(R, \sT, \ba^t)$ is strongly $F$-regular by the original definition of Hochster and Huneke.
\end{remark}

\begin{remark}
\label{RemIf1InThen1InAgain}
If $\phi \in \sT_e \cdot \left(F^e_* \ba^{\lceil t(p^e - 1) \rceil} \right)$ is such that $1 \in \phi(F^e_* (dR))$.  Then $1 \in \phi^n(F^{ne}_* (dR))$ for all $n > 0$.  This is because we have the containments:
\[
\begin{split}
1 \in \phi(F^e_* (dR) ) \\
\subseteq \phi(F^e_* \phi(F^e_* (dR) )) \\
= \phi^2(F^{2e}_* (dR) ) \\
\subseteq \ldots .
\end{split}
\]
Note that $\phi^n \in \sT_{ne} \cdot \left(F^{ne}_* \ba^{\lceil t(p^{ne} - 1) \rceil} \right)$.
\end{remark}

\begin{definition} \cite{MehtaRamanathanFrobeniusSplittingAndCohomologyVanishing}, \cite{SchwedeCentersOfFPurity}
Given a triple $(R, \sT, \ba^t)$, we say that an ideal $J \subseteq R$ is \emph{uniformly $(\sT, \ba^t, F)$-compatible} if for all $e \geq 0$ and all homogeneous $\phi \in \sT_e$, we have that $\phi(F^e_* \ba^{\lceil t(p^e - 1) \rceil} J) \subseteq J$.  Equivalently, for all $e \geq 0$ and all $\phi \in \sT_e \cdot \left(F^e_* \ba^{\lceil t(p^e - 1) \rceil} \right)$, we can require that $\phi(F^e_* J) \subseteq J$.  Another equivalent definition would be to require that for all $e \geq 0$, all $\phi \in F^e_* \sT_e$ and all $a \in \ba^{\lceil t(p^e - 1) \rceil}$, we have that $\phi(F^e_* aJ) \subseteq J$.
For the definition for pairs $(R, \sT)$, set $\ba = R$.
\end{definition}

\begin{definition} \cite{HochsterHunekeTC1}, \cite{LyubeznikSmithCommutationOfTestIdealWithLocalization}, \cite{HochsterFoundations}, \cite{SchwedeCentersOfFPurity}
\label{DefnBigTestIdeal}
The \emph{big test ideal $\tau_b(R; \sT, \ba^t)$} of a triple $(R, \sT, \ba^t)$, if it exists, is the unique smallest ideal $J$ that satisfies two conditions:
\begin{itemize}
\item[(1)]  $J$ is uniformly $(\sT, \ba^t, F)$-compatible, and
\item[(2)]  $J \cap R^{\circ} \neq \emptyset$.
\end{itemize}
The \emph{big test ideal $\tau_b(R; \sT)$} of a pair $(R, \sT)$ is defined by setting $\ba = R$.
\end{definition}

\begin{remark}
 For a triple $(X, \sT, \ba^t)$ constructed from $(X, \Delta, \ba^t)$ as in Remark \ref{RemCanConstructNewAlgebras}, it is obvious from the definition and from \cite[Lemma 2.1]{HaraTakagiOnAGeneralizationOfTestIdeals} or \cite[Theorem 6.3]{SchwedeCentersOfFPurity} that the test ideal of $(X, \Delta, \ba^t)$ agrees with the test ideal of $(X, \sT, \ba^t)$ where $\sT$ is the algebra constructed from the pair $(X, \Delta)$ as above.

Furthermore, given a triple $(R, \sT, \ba^t)$, one can likewise construct a pair $(R, \sT')$ which has the same test ideal as the triple $(R, \sT, \ba^t)$.  Simply define $\sT'_i := \sT_i \cdot \left(F^i_* \ba^{\lceil t(p^i - 1) \rceil} \right)$.
\end{remark}

We will show that the big test ideal exists under the assumption that $(R, \sT)$ is \reduced.

\begin{theorem}
\label{ThmTestIdealExists}
Suppose that $(R, \sT)$ is a \reduced{} pair (respectively, that $(R, \sT, \ba^t)$ is a \reduced{} triple).  Then the big test ideal $\tau_b(R; \sT)$ (respectively $\tau_b(R; \sT, \ba^t)$) exists.
\end{theorem}

In order to prove this we need several preliminary results.  But first we give a definition.

\begin{definition}
 Suppose that $(R, \sT, \ba^t)$ is a \reduced{} triple.  An element $c \in \ba \cap R^{\circ}$ is called a \emph{\homgen{} big sharp test element} (or simply a \emph{\homgenbs{} test element}) for the triple $(R, \sT, \ba^t)$ if for every $d \in R^{\circ}$ there exists a $\phi \in \sT_i$, $i > 0$ such that $c \in \phi(F^i_* \ba^{\lceil t(p^i - 1) \rceil} (d))$.  One can also make an analogous definition for pairs.
\end{definition}

\begin{remark}
 If $c$ is a \homgenbs{} test element, then for any $c' \in R^{\circ}$, $c' c$ is also a \homgenbs{} test element.
\end{remark}

Once we prove that \homgenbs{} test elements exist, Theorem \ref{ThmTestIdealExists} follows quickly.

\begin{proof}[Proof of Theorem \ref{ThmTestIdealExists} modulo \homgenbs{} test elements]
 Suppose that $c \in \ba \cap R^{\circ}$ is a \homgenbs{} test element for a \reduced{} triple $(R, \sT, \ba^t)$, and that $J$ is uniformly $(\sT, \ba^t, F)$-compatible with $J \cap R^{\circ} \neq \emptyset$.  Then clearly $c \in J$.  On the other hand, it is easy to see that the sum
\[
 I := \sum_{e \geq 0} \sum_{\phi \in \sT_e} \phi(F^e_* \ba^{t(p^e - 1)} (c)) \subseteq J
\]
is the smallest uniformly $F$-compatible ideal containing $c$.  Therefore we obtain that $I = \tau_b(R; \sT, \ba^t)$.
\end{proof}

Therefore, we will prove that a \homgenbs{} test element exists.  The idea is exactly the same as the usual construction of test elements, see \cite[Section 6]{HochsterHunekeTC1}, \cite[Lemma 2.5]{TakagiInterpretationOfMultiplierIdeals}, \cite{HochsterFoundations}, \cite[Section 6]{SchwedeFAdjunction}.

\begin{proposition}
\label{PropHDBSTestEltsExist}
Suppose that $(R, \sT, \ba^t)$ is a \reduced{} triple with \reduced{} homogeneous element $\phi \in \sT_e$.  Then there exists a \homgenbs{} test element for $(R, \sT, \ba^t)$.
\end{proposition}

We first need a lemma.

\begin{lemma}
\label{InitialChoiceOfC}
Assuming the hypotheses of Proposition \ref{PropHDBSTestEltsExist}.  Suppose that $c \in \ba \cap R^{\circ}$ is such that $R_c$ is strongly regular (for example, if $R_c$ regular) and $$\vlocus(\phi) \subseteq V(c) \subseteq X = \Spec R.$$  Then $(R_c, \sT_c, \ba_c^t)$ is strongly $F$-regular.  Furthermore, for every $d \in R^{\circ}$, there exists positive integers $m,n > 0$ such that $c^m \in \phi^n(F^{ne}_* (d) \ba^{\lceil t(p^{ne} - 1) \rceil})$.
\end{lemma}
\begin{proof}
We see that $\overline{\phi}$ generates $\Hom_{R_c}(F^e_* R_c, R_c)$ as an $F^e_* R_c$-module by the hypothesis about the degeneracy locus of $\phi$.  This then implies that $\overline{\phi^n}$ generates $\Hom_{R_c}(F^{ne}_* R_c, R_c)$ as an $F^{ne}_* R_c$-module (since $R_c$ is regular, see \cite[Lemma 3.8, Corollary 3.9]{SchwedeFAdjunction}).   In other words, $(\sT_c)_{ne} = (\sC(R_c))_{ne}$.  Therefore, since $R_c$ is strongly $F$-regular, there exists some $n > 0$ and $\psi \in \Hom_{R_c}(F^{ne}_* R_c, R_c)$ such that $1 \in \psi(F^{ne}_* (\overline d) \ba_c^{\lceil t(p^{ne} - 1) \rceil})$ (note $\ba_c = R$).  But then
\[
1 \in \overline{\phi}^n(F^{ne}_* (\overline d) \ba_c^{\lceil t(p^{ne} - 1) \rceil})
\]
as desired since $\psi$ is obtained from $\overline \phi^n$ by pre-multiplication by an element of $R$.

By clearing denominators, the second result is obtained.
\end{proof}

\begin{proof}[Proof of Proposition \ref{PropHDBSTestEltsExist}]
Fix notation as in the statement of Lemma \ref{InitialChoiceOfC}.  In particular, choose some \reduced{}  $\phi$.  Choose $c$ as in Lemma \ref{InitialChoiceOfC}.  Setting $d = 1$ in Lemma \ref{InitialChoiceOfC}, we see that there exists some positive integers $m_1, n_1$ so that $c^{m_1} \in \phi^{n_1}(F^{n_1 e}_* (1) \ba^{\lceil t(p^{n_1 e} - 1) \rceil})$.  Furthermore, we have
\[
\begin{split}
c^{2m_1} \in c^{m_1}  \phi^{n_1}\left(F^{n_1 e}_* (1) \ba^{\lceil t(p^{n_1 e} - 1) \rceil}\right)\\
= \phi^{n_1}\left(F^{n_1 e}_* (c^{p^{n_1 e} m_1}) \ba^{\lceil t(p^{n_1 e} - 1) \rceil}\right)\\
\subseteq \phi^{n_1}\left(F^{n_1 e}_* (c^{2m_1}) \ba^{\lceil t(p^{n_1 e} - 1) \rceil}\right)\\
\subseteq \phi^{n_1}\left(F^{n_1 e}_* \ba^{\lceil t(p^{n_1 e} - 1) \rceil} c^{m_1}  \phi^{n_1}\left(F^{n_1 e}_* (1) \ba^{\lceil t(p^{n_1 e} - 1) \rceil}\right) \right)\\
\subseteq \phi^{2 n_1}\left(F^{2 n_1 e}_* (c^{2m_1}) \ba^{\lceil t(p^{2n_1 e} - 1) \rceil} \right) \\
\subseteq \ldots \subseteq \phi^{n n_1}\left(F^{n n_1 e}_* (c^{2m_1}) \ba^{\lceil t(p^{n n_1 e} - 1) \rceil} \right)
\end{split}
\]
for every integer $n > 0$.  Therefore $c^{2m_1} \in \phi^{n n_1}\left(F^{n n_1 e}_* \ba^{\lceil t(p^{n n_1 e} - 1) \rceil} \right)$ for every integer $n > 0$.  We will show that $c^{3m_1}$ is a \homgenbs{} test element.

Now fix any $d \in R^{\circ}$.  Again by Lemma \ref{InitialChoiceOfC}, we can find some positive integers $m_d, n_d$ such that $c^{m_d} \in \phi^{n_d}(F^{n_d e}_* (d) \ba^{\lceil t(p^{n_d e} - 1) \rceil})$.  If $m_d$ is less than $3m_1$, then we are done.  Otherwise, we may assume that $m_d = p^{nn_1e} m_1$ for some $n > 0$ (since making $m_d$ larger is harmless).  But then
\[
\begin{split}
c^{3m_1} \in c^{m_1} \phi^{n n_1}\left(F^{n n_1 e}_* \ba^{\lceil t(p^{n n_1 e} - 1) \rceil} \right)\\
= \phi^{n n_1}\left(F^{n n_1 e}_*  c^{p^{nn_1e} m_1} \ba^{\lceil t(p^{n n_1 e} - 1) \rceil} \right) \\
\subseteq \phi^{n n_1}\left(F^{n n_1 e}_* \ba^{\lceil t(p^{n n_1 e} - 1) \rceil} \phi^{n_d}(F^{n_d e}_* (d) \ba^{\lceil t(p^{n_d e} - 1) \rceil}) \right) \\
\subseteq \phi^{n n_1 + n_d}\left( F^{n n_1 e + n_d e}_* (d) \ba^{\lceil t(p^{n n_1 e + n_d e} - 1) \rceil} \right)
\end{split}
\]
which completes the proof.
\end{proof}

We now list some basic properties of test ideals.

\begin{proposition}
Suppose that $(R, \sT, \ba^t)$ is a \reduced{} triple.  Further suppose that $\sT' \subseteq \sT$ is a graded subalgebra such that the triple $(R, \sT', \ba^t)$ is also \reduced{}.  Then the following hold:
\begin{itemize}
\item[(i)]  $\tau_b(R; \sT, \ba^t) \supseteq \tau_b(R; \sT', \ba^t)$
\item[(ii)] For any multiplicative set $W$, we have
\[
W^{-1} \tau_b(R; \sT, \ba^t) = \tau_b(W^{-1}R; W^{-1}\sT, W^{-1}\ba^t).
\]
\item[(iii)]  $(R, \sT, \ba^t)$ is strongly $F$-regular if and only if $\tau_b(R; \sT, \ba^t) = R$.
\end{itemize}
\end{proposition}
\begin{proof}
Part (i) is obvious.  Part (ii) is easy once we observe that a \homgenbs{} test element remains a \homgenbs{} test element after localization.  But this follows since if $W$ is a multiplicative set, every element of $(W^{-1} R)^{\circ}$ can be written as a fraction $r/w$ for $w \in W$ and $r \in R^{\circ}$, see for example \cite[Page 57]{HochsterFoundations}.

The $(\Rightarrow)$ direction of (iii) is obvious.  Thus we prove the $(\Leftarrow)$ direction.  Since we can absorb the $\ba^t$ term into $\sT$ as in Remark \ref{RemCanConstructNewAlgebras}, we may assume that $\ba = R$.  Fix $d \in R^{\circ}$, choose $c$ to be a \homgenbs{} test element and write
\[
R = \tau_b(R; \sT) = \sum_e \sum_{\phi \in \sT_e} \phi(F^e_* (cR)).
\]
First we do the case where $R$ is local with maximal ideal $\bm$.  In that case, since a sum of ideals contained in $\bm$ is still contained in $\bm$, there must be some $e > 0$ and $\phi \in \sT_e$ such that $\phi(F^e_* (cR)) \nsubseteq \bm$ which implies that $1 \in \phi(F^e_* (cR))$.  Thus since for every $d \in R^{\circ}$ we have that $c \in \psi(F^i_* (dR))$ for some $i > 0$ and $\psi \in \sT_i$, we have $1 \in (\phi \cdot \psi)(F^{e+i}_* (dR))$ as desired.

Now assume that $R$ is no longer necessarily local.  Choose some maximal ideal $\bm$.  By (ii) and the above work, we see that there exists some $\phi \in \sT_e$ such that $\phi(F^e_* (dR))_{\bm} = R_{\bm}$.  Thus this also holds in a neighborhood $U = \Spec R_b$ of $\bm$.  We can cover $\Spec R$ by a finite number of such neighborhoods $U_j = \Spec R_{b_j} = \Spec R[b_j^{-1}]$ with associated $\phi_j \in \sT_{e_j}$ such that $\overline \phi_j(F^{e_j}_* dR_{b_j}) = R_{b_j}$.  By replacing $\phi_j$ with self-compositions, and using Remark \ref{RemIf1InThen1InAgain}, we may assume that all the $e_j$ are equal to the same $e$.  Consider $I := \sum_{j = 1}^m \phi_j(F^{e}_* (dR)) \subseteq R$.    Now $I_{b_j} = R_{b_j}$ by construction.  This implies that $I := \sum_{j = 1}^m \phi_j(F^{e}_* (dR)) = R$.

Note that by Remark \ref{RemarkNaturalMapsToR}, there is a natural map $\sT_e \rightarrow R$ (the evaluation map).  Thus we also have a natural map $\Phi : \sT_e \cdot \left(F^e_* (dR) \right) \rightarrow R$ by restriction.  The image of $\Phi$ certainly contains $I$.  In particular, there exists a $\phi \in \sT_e$ such that $1 \in \phi(F^e_* (dR))$ as desired.
\end{proof}

\section{log-$\bQ$-Gorenstein pairs}

In this section, we consider special pairs $(R, \sT)$ that behave essentially like pairs $(R, \Delta)$ where $R$ is normal and $K_R + \Delta$ is $\bQ$-Cartier with index not divisible by $p > 0$.

\begin{definition}
\label{DefinitionFreeGen}
A \reduced{} pair $(R, \sT)$ (or \reduced{} triple $(R, \sT, \ba^t)$) will be called \emph{\freegen{}} if there exists an integer $e_0 > 0$ such that the following holds:
\begin{itemize}
\item[(i)]  $\sT_{e_0}$ is isomorphic (as an $F^{e_0}_* R$-module) to $F^{e_0}_* R$.
\item[(ii)]  For every $e > 0$, write $e = ne_0 + r$ for some $n, r \geq 0$ (note that we do not require $r < e_0$).  Then the natural map
\[\xymatrix@R=6pt{
\sT_{ne_0} \tensor_{F^{ne_0}_* R} F^{ne_0}_* \sT_{r} \ar[r] & \sT_e\\
\phi \tensor F^{ne_0}_* \psi \ar@{|->}[r] & \phi \circ (F^{ne_0}_* \psi)
}
\]
induced by composition is an isomorphism of $F^e_* R$-modules.
\end{itemize}
\end{definition}

\begin{remark}
Note that (i) and (ii) imply that $\sT_{me_0} \cong F^{me_0}_* R$ for all $m \geq 0$.
\end{remark}

\begin{remark}
Given a \reduced{} $\phi \in \sC(R)_i$, the subalgebra $\sT = R\langle \phi \rangle$ generated by $\phi$ and $R = \sC(R)_0$ is \freegen.
\end{remark}

\begin{remark}
It follows from \cite[Corollary 3.9]{SchwedeFAdjunction} that if $R$ is a normal domain and $\Delta$ is an effective $\bQ$-divisor such that $R((p^{e_0} - 1)(K_R + \Delta))$ is free\footnote{This always happens locally for pairs $(X, \Delta)$ such that $K_X + \Delta$ is $\bQ$-Cartier with index not divisible by $p > 0$. It follows because if $(p^e - 1)\Delta$ is integral, then $\Hom_R(F^e_* R((p^e - 1)\Delta), R) \cong R((1-p^e)(K_X + \Delta))$.}, then the associated pair $(R, \sT)$ is \freegen{}.
\end{remark}

\begin{remark}
Note that if $(R, \sT)$ is \freegen{}, then $\sT$ is also finitely generated as an algebra (over $\sT_0$).
\end{remark}

\begin{proposition} \cite[Lemma 8.16]{HochsterHunekeTC1}, \cite[Proposition 3.5]{TakagiPLTAdjoint}, \cite[Proposition 4.7]{SchwedeFAdjunction}
\label{PropCanUseACertainSubset}
Suppose that $(R, \sT, \ba^t)$ is \freegen{} with associated $e_0 > 0$.  Consider the graded subalgebra
\[
\sT' = \oplus_{n \geq 0} \sT_{ne_0} \subseteq \sT.
\]
Then $\tau_b(R; \sT, \ba^t) = \tau_b(R; \sT', \ba^t)$.
\end{proposition}
\begin{proof}
We may select $c \in \ba \cap R^{\circ}$ which is a \homgenbs{} test element for $(R, \sT', \ba^t)$.  Choose $d \in R^{\circ}$ such that
\[
d \ba^{\lceil t(p^{ne_0 + k} - 1) \rceil} \subseteq (\ba^{\lceil t(p^{n{e_0}} - 1) \rceil})^{[p^k]}
 \]
for all $n \geq 0$ and all $k < e_0$ (we can do this due to \cite[Lemma 4.6, Proposition 4.7]{SchwedeFAdjunction}).  Then $d c^{p^{e_0}}$ is also a \homgenbs{} test element for both $(R, \sT', \ba^t)$ and $(R, \sT, \ba^t)$.  Fix $\psi \in \sT_{e_0}$ to be a generator of $\sT_{e_0}$ as an $F^{e_0}_* R$-module.  Then $\psi^n$ is a generator of $\sT_{ne_0}$ as an $F^{ne_0}_* R$-module.  Furthermore, if $e = ne_0 + k$, then every element $\phi \in \sT_e$ can be written as $\psi^n \circ F^{ne_0}_* \phi'$ for some $\phi' \in \sT_{k}$.  Therefore,
\[
\begin{split}
\tau_b(R; \sT, \ba^t) = \sum_{e \geq 0} \sum_{\phi \in \sT_e} \phi(F^e_* (dc^{p^{e_0}}) \ba^{\lceil t(p^e - 1) \rceil}) \\
= \sum_{n \geq 0} \sum_{k = 0}^{e_0 - 1} \sum_{\phi \in \sT_{ne_0 + k}} \phi(F^{ne_0 + k}_* (dc^{p^{e_0}}) \ba^{\lceil t(p^{ne_0 + k} - 1) \rceil}) \\
= \sum_{n \geq 0} \sum_{k = 0}^{e_0 - 1} \sum_{\phi' \in \sT_{k}} \psi^n(F^{ne_0}_* \phi'(F^{k}_* (dc^{p^{e_0}}) \ba^{\lceil t(p^{ne_0 + k} - 1) \rceil})) \\
\subseteq \sum_{n \geq 0} \sum_{k = 0}^{e_0 - 1} \sum_{\phi' \in \sT_{k}} \psi^n(F^{ne_0}_* \phi'(F^{k}_* (c^{p^{k}}) (\ba^{\lceil t(p^{n{e_0}} - 1) \rceil})^{[p^k]})) \\
= \sum_{n \geq 0} \sum_{k = 0}^{e_0 - 1} \sum_{\phi' \in \sT_{k}} \psi^n(F^{ne_0}_* (c) (\ba^{\lceil t(p^{n{e_0}} - 1) \rceil}) \phi'(F^{k}_*  (R)))\\
\subseteq \sum_{n \geq 0} \sum_{k = 0}^{e_0 - 1} \sum_{\phi' \in \sT_{k}} \psi^n(F^{ne_0}_* (c) (\ba^{\lceil t(p^{n{e_0}} - 1) \rceil}) R)\\
= \sum_{n \geq 0} \psi^n(F^{ne_0}_* (c) (\ba^{\lceil t(p^{n{e_0}} - 1) \rceil}))
= \tau_b(R; \sT', \ba^t).
\end{split}
\]
Note that the final equality occurs because $\psi^n$ generates $\sT_{ne}$ as an $F^{ne}_* R$-module.
\end{proof}

\begin{remark}
This result is closely related to the fact that if $c z^{p^e} \in I^{[p^e]}$ for all $e = n e_0$, then $z \in I^*$, compare with \cite[Proposition 4.7]{SchwedeFAdjunction} \cite[Proposition 3.5]{TakagiPLTAdjoint} and \cite[Lemma 8.16]{HochsterHunekeTC1}.  In fact, those previous results all had related proofs.  However they all relied on (Matlis-dual) analogs of the fact that there are natural maps $\Hom_R(F^{e+k}_* R, R) \rightarrow \Hom_R(F^e_* R, R)$.  We took a slightly different approach above in the proof of Proposition \ref{PropCanUseACertainSubset}.
\end{remark}

We also link the test ideals of \freegen{} triples $(R, \sT, \ba^t)$ with the test ideals of triples $(R, \Delta, \ba^t)$.

\begin{corollary}
\label{CorFreegenIsLogQGor}
Suppose that $R$ is a normal domain and that $(R, \sT, \ba^t)$ is a \freegen{} triple with associated $e_0$.  Then there exists an effective $\bQ$-divisor $\Delta$ such that $(p^{e_0} - 1)(K_R + \Delta)$ is Cartier and such that
\[
\tau_b(R; \Delta, \ba^t) = \tau_b(R; \sT, \ba^t).
\]
\end{corollary}
\begin{proof}
Choose $\phi \in \sT_{e_0}$ which generates $\sT_{e_0}$ as an $F^{e_0}_* R$-module.  Then by \cite[Theorem 3.10]{SchwedeFAdjunction}, $\phi$ corresponds to an effective $\bQ$-divisor $\Delta$.  The result then follows immediately from Definition \ref{DefnBigTestIdeal}, Proposition \ref{PropCanUseACertainSubset} and \cite[Proposition 4.7]{SchwedeFAdjunction}.
\end{proof}

\section{Proof of the main theorem}

In this section we prove our main result, which is stated below.

\begin{theorem}
\label{ThmMain}  Suppose that $(R, \sT, \ba^t)$ is a \ureduced{} triple.  Then there exist graded subalgebras $\sT^1, \ldots, \sT^n \subseteq \sT$ such that each $(R, \sT^i, \ba^t)$ is \freegen{} and such that
\[
\tau_b(R; \sT, \ba^t) = \sum_{i = 1}^n \tau_b(R; \sT^i, \ba^t).
\]
\end{theorem}
We obtain the following corollaries.
\begin{corollary}
\label{CorMain}
Suppose that $(R, \ba^t)$ is a pair and that $R$ is a $F$-finite normal domain.  Then
\[
\tau_b(R; \ba^t) = \sum_{\Delta} \tau_b(R; \Delta, \ba^t)
\]
where the sum is over effective $\bQ$-divisors $\Delta$ such that $K_X + \Delta$ is $\bQ$-Cartier with index not divisible by $p > 0$ and $X = \Spec R$.
\end{corollary}
\begin{proof}
The containment $\supseteq$ follows from the fact that $\tau(R; \Delta, \ba^t) \subseteq \tau_b(R; \ba^t)$ for every effective $\Delta$.  The containment $\subseteq$ follows from Corollary \ref{CorFreegenIsLogQGor}.
\end{proof}

\begin{remark}
The actual correspondence from \cite{SchwedeFAdjunction} between $\phi$'s and $\Delta$'s identifies $\phi$ (modulo some equivalence relation) with $\Delta$ such that $(p^e - 1)(K_X + \Delta)$ is Cartier and that $\O_X((pe - 1)(K_X + \Delta)) \cong \O_X$.  Thus we may further restrict our sum in the previous corollary to be over such $\Delta$.
\end{remark}

\begin{corollary}
\label{CorDontNeedBAT}
Suppose that $(R, \ba^t)$ is a pair and that $R$ is an $F$-finite normal domain, $X = \Spec R$.  Then there exists finitely many effective $\bQ$-divisors $\Delta_i$ such that $K_X + \Delta_i$ is $\bQ$-Cartier with index not divisible by $p$ and that
\[
\tau_b(R; \ba^t) = \sum_{\Delta_i} \tau_b(R; \Delta_i).
\]
\end{corollary}
\begin{proof}
Out of the pair $(R, \ba^t)$, one may create a graded subalgebra $\sT \subseteq \sC(R)$ which has the same test ideal as the pair $(R, \ba^t)$, see Remark \ref{RemCanConstructNewAlgebras}.  Then one can apply Theorem \ref{ThmMain} and Corollary \ref{CorFreegenIsLogQGor}.
\end{proof}

The proof of Theorem \ref{ThmMain} is rather technical.  So we first outline the strategy we will use.
There are three steps.
\begin{itemize}
\item[(i)]  Find \reduced{} homogeneous $\phi_i \in \sT$, $i = 1, \dots, m$ such that if $\sT'$ is the subalgebra of $\sT$ generated by $\sT_0$ and the $\phi_i$, then $\tau_b(R; \sT, \ba^t) = \tau_b(R; \sT', \ba^t)$.  This can be done by the Noetherian property of $R$.
\item[(ii)]  Find an element $d \in R^{\circ}$ that is simultaneously a \homgenbs{} test element for all triples $(R, \sS, \ba^t)$ where $\sS = R\langle\psi\rangle$ ranges over all \freegen{} subalgebras of $\sT$ generated by $\sT_0$ and a single product $\psi$ of the $\phi_i$.
\item[(iii)]  Observe that $\tau_b(R; \sT, \ba^t)$ can be generated elements which are contained in the test ideals $\tau_b(R; \sS, \ba^t)$ for various $\sS$ as in step (ii).
\end{itemize}

We first prove step (i).

\begin{proposition}
\label{PropCanReplaceByFiniteGenSet}
 Suppose that $(R, \sT, \ba^t)$ is a \ureduced{} triple.  Then there exists a subalgebra $\sT'$ generated by $R \cong \sT_0$ and finitely many additional homogeneous $\phi_i$ such that $\tau_b(R; \sT, \ba^t) = \tau_b(R; \sT', \ba^t)$.  Furthermore, we can assume that the $\phi_i$'s are non-zero at every minimal prime of $R$ and thus that $(R, \sT', \ba^t)$ is also \ureduced.
\end{proposition}
\begin{proof}
Begin by choosing any homogeneous $\phi_1$ that is non-zero at every minimal prime of $R$ and set $\sT^1$ to be the subalgebra of $\sT$ generated by $R$ and $\phi_1$.  Note that $\tau_b(R; \sT^1, \ba^t)$ is the smallest ideal $J$ such that $J \cap R^{\circ} \neq \emptyset$ and such that $\phi(F^e_* \ba^{\lceil t(p^e - 1) \rceil}) \subseteq J$ for all $e \geq 0$ and $\phi \in \sT^1_e$.  Therefore, since $\sT^1 \subseteq \sT$, $\tau_b(R; \sT^1, \ba^t) \subseteq \tau_b(R; \sT, \ba^t)$.  If we have equality, we are done, set $\sT' = \sT^1$.  Otherwise choose some homogeneous $\phi_2 \in \sT_{e_2}$, such that $\phi_2(F^{e_2}_* \tau_b(R; \sT^1, \ba^t)) \nsubseteq \tau_b(R; \sT^1, \ba^t)$.  We also assume that $\phi_2$ is non-zero at every minimal prime at $R$ (note that this is possible since $(R, \sT, \ba^t)$ is \ureduced).  Set $\sT_2$ to be the subalgebra of $\sT$ generated by $\sT_1$ and $\phi_2$.

Note that by hypothesis, $\tau_b(R; \sT^1, \ba^t) \subsetneq \tau_b(R; \sT^2, \ba^t)$.  Again, if $\tau_b(R; \sT^2, \ba^t)$ is equal to $\tau_b(R; \sT, \ba^t)$, we are done, set $\sT' = \sT^2$.  Otherwise, we can continue the process.  However, this must stop eventually since $R$ is Noetherian.
\end{proof}

We now need to find an element that is simultaneously a \homgenbs{} test element for all triples $(R, R\langle \psi \rangle, \ba^t)$ as in (ii).

\begin{proposition}
\label{PropUniformExistenceOfTestElements}
Suppose that $(R, \sT', \ba^t)$ is a triple and that $\sT'$ is generated by finitely many homogeneous $\phi_i$ each of which are non-zero at all of the minimal primes of $R$.  We will use $\phi_{i_1, \ldots, i_n}$ to denote the product $\phi_{i_1} \cdot \dots \cdot \phi_{i_n}$.  Set $\sS^{i_1, \ldots, i_n} = R\langle\phi_{i_1, \ldots, i_n}\rangle$ to be the subalgebra of $\sT'$ generated by $R$ and $\phi_{i_1, \ldots, i_n}$ and suppose that $c_i \in \ba \cap R^{\circ}$ is a \homgenbs{} test element for $(R, \sS^i, \ba^t)$.  Then there exists a single $c \in \ba \cap R^{\circ}$ which is a \homgenbs{} test element for every triple $(R, \sS^{i_1, \ldots, i_n}, \ba^t)$.
\end{proposition}

Before proving this we need a lemma.

\begin{lemma}
\label{LemConditionForTestElement}
Suppose that $(R, R\langle\phi\rangle, \ba^t)$ is a triple where $\phi$ is a homogeneous element of $\Hom_R(F^e_* R, R)$.  Further suppose that $c \in \ba \cap R^{\circ}$ is an element such that $R_c$ is regular, $\vlocus(\phi) \subseteq V(c)$ and that $c \in \phi(F^e_* (c) \ba^{\lceil t(p^e - 1) \rceil} )$.  Then $c^2$ is a \homgenbs{} test element for $(R, R\langle\phi\rangle, \ba^t)$.
\end{lemma}
\begin{proof}
Choose $d \in R^{\circ}$.  Since $R_c$ is regular and $\vlocus(\phi) \subseteq V(c)$ we have that $R_c$ is strongly $F$-regular and $\overline \phi$ generates $\Hom_{R_c}(F^e_* R_c, R_c)$ as an $F^e_* R_c$-module, see Lemma \ref{InitialChoiceOfC}.  Therefore there exists an $n$ so that $1 \in \overline{\phi^n}(F^{ne}_* dR_c )$. By clearing denominators we see that $c^m \in \phi^n(F^{ne}_* (d) \ba^{\lceil t(p^{ne} - 1) \rceil} )$ (note $c$ was in $\ba$).  Roughly speaking, our goal is now to reduce $m$ to a number which is independent of $d$.

On the other hand, by induction, we claim for all $l \geq 1$ that $$c \in \phi^l(F^{le}_* (c) \ba^{\lceil t(p^{le} - 1) \rceil}).$$  This is because
\[
\begin{split}
\phi^l\left(F^{le}_* (c) \ba^{\lceil t(p^{le} - 1) \rceil}\right) \\
\subseteq \phi^l\left(F^{le}_* \ba^{\lceil t(p^{le} - 1) \rceil} \phi(F^e_* (c) \ba^{\lceil t(p^e - 1) \rceil} ) \right) \\
= \phi^l\left(F^{le}_* \phi(F^e_* (c) (\ba^{\lceil t(p^{le} - 1) \rceil})^{[p^e]} \ba^{\lceil t(p^e - 1) \rceil} ) \right) \\
\subseteq \phi^{l+1}\left( F^{(l+1)e}_* (c) (\ba^{\lceil t(p^{(l+1)e} - 1) \rceil}) \right)
\end{split}
\]
We now show that there exists an $l$ such that $c^2 \in \phi^{l}(F^{le}_* (c^m) \ba^{\lceil t(p^{le} - 1) \rceil})$.  Choose $l$ such that $p^{le} + 1 \geq m$.  Then
\[
\begin{split}
c^2 \in c \cdot \phi^l(F^{le}_* (c) \ba^{\lceil t(p^{le} - 1) \rceil}) \subseteq \phi^l(F^{le}_* (c^{p^{le} + 1}) \ba^{\lceil t(p^{le} - 1) \rceil}) \\
\subseteq \phi^l(F^{le}_* (c^{m}) \ba^{\lceil t(p^{le} - 1) \rceil})
\end{split}
\]
as desired.  But then we have that
\[
\begin{split}
c^2 \in \phi^l(F^{le}_* c^{m} \ba^{\lceil t(p^{le} - 1) \rceil}) \\
\subseteq \phi^l\left(F^{le}_* \ba^{\lceil t(p^{le} - 1) \rceil} \phi^n(F^{ne}_* (d) \ba^{\lceil t(p^{ne} - 1) \rceil} )\right) \\
= \phi^{l+n}\left( F^{(l+n)e}_*  (d) (\ba^{\lceil t(p^{le} - 1) \rceil})^{[p^{ne}]} \ba^{\lceil t(p^{ne} - 1) \rceil} )\right) \\
\subseteq \phi^{l+n}\left( F^{(l+n)e}_*   (d) \ba^{\lceil t(p^{(l+n)e} - 1) \rceil} )\right)
\end{split}
\]
which completes the proof.
\end{proof}

\begin{proof}[Proof of Proposition \ref{PropUniformExistenceOfTestElements}]
For each $\phi_i$, we claim we can choose $c_i \in \ba \cap R^{\circ}$ such that the following conditions hold:
\begin{itemize}
\item[(a)]  $R_{c_i}$ is regular.
\item[(b)]  $\vlocus(\phi_i) \subseteq V(c_i)$.
\item[(c)]  $c_i$ is a \homgenbs{} test element for $(R, \sS^i, \ba^t)$.
\item[(d)]  $c_i \in \phi_i(F^{e_i}_* (c_i)\ba^{\lceil t(p^{e_i}-1) \rceil})$.
\end{itemize}
Notice that finding a $c_i$ which satisfies (a), (b) and (c) is easy.  Therefore, fix a $c_i'$ satisfying the first three conditions.  To find a $c_i$ also satisfying condition (d), first we note that there exists some $m > 0$ such that
\begin{equation}
\label{EqnCtoMInside}
c_i'^m \in \phi_i(F^{e_i}_* (c_i')\ba^{\lceil t(p^{e_i}-1) \rceil}).
\end{equation}
To see this, localize at $c_i'$.  Then $R_{c_i'}$ is $F$-pure so that $1 \in \psi(F^{e_i}_* R_{c_i'})$ for some $\psi \in \sC(R_{c_i'})_{e_i}$.  But any such $\psi$ can be expressed as $\overline \phi_i$ pre-composed with some element of $R_{c_i'}$.   Therefore $1 \in \overline \phi_i(F^{e_i}_* R_{c_i'})$.  Clearing denominators proves that Equation (\ref{EqnCtoMInside}) holds for some $m$.  Note that this is a slight improvement over Lemma \ref{InitialChoiceOfC} since we need not raise $\phi_i$ to a power.
We then have that
\[
\begin{split}
c_i'^{2m}
\in c_i'^m \phi_i(F^{e_i}_* (c_i')\ba^{\lceil t(p^{e_i}-1) \rceil}) \\
= \phi_i(F^{e_i}_* (c_i'^{1 + mp^{e_i}}) \ba^{\lceil t(p^{e_i}-1) \rceil})\\
 \subseteq \phi_i(F^{e_i}_* (c_i'^{2m}) \ba^{\lceil t(p^{e_i}-1) \rceil}).
\end{split}
\]
Set $c_i$ to be $c_i'^{2m}$.

Also notice that if $c_i$ satisfies condition (d), then for every $d \in R$, $dc_i$ also satisfies condition (d) since
\[
\begin{split}
dc_i \in d \phi_i(F^{e_i}_* (c_i)\ba^{\lceil t(p^{e_i}-1) \rceil})\\
 = \phi_i(F^{e_i}_* (d^{p^e} c_i)\ba^{\lceil t(p^{e_i}-1) \rceil}) \\
\subseteq \phi_i(F^{e_i}_* (d c_i)\ba^{\lceil t(p^{e_i}-1) \rceil}).
\end{split}
\]
Set $c' = \prod_i c_i$.  Consider $\phi_{i_1, \dots, i_n} = \phi_{i_1} \cdot \dots \cdot \phi_{i_n}$.  Set $U = \Spec R_{c'} \subseteq \Spec R$.  Note that $U$ is regular.  Also note that at each (possibly non-closed) point $Q \in U$, $\overline{\phi_i}$ generates $\Hom_{R_Q}(F^{e_i}_* R_Q, R_Q)$ as an $F^{e_i}_* R_Q$-module.  Therefore $\overline{\phi_{i_1, \dots, i_n}}$ also generates $\Hom_{R_Q}(F^{e_{i_1} + \dots + e_{i_n}}_* R_Q, R_Q)$ as an $F^{e_{i_1} + \dots + e_{i_n}}_* R_Q$-module (this follows from \cite[Lemma 3.8]{SchwedeFAdjunction}).  In particular, $\vlocus(\phi_{i_1, \dots, i_n}) \subseteq V(c')$.

Finally, note that we have
\[
\begin{split}
c' \in \phi_{{i_1}}(F^{e_{i_1}}_* \ba^{\lceil t(p^{e_{i_1}} - 1) \rceil} (c') ) \\
\subseteq \phi_{{i_1}}(F^{e_{i_1}}_* \ba^{\lceil t(p^{e_{i_1}} - 1) \rceil} \phi_{e_{i_2}}(F^{e_{i_2}}_* \ba^{\lceil t(p^{e_{i_2}} - 1) \rceil} (c') ) )\\
 \subseteq \phi_{{i_1},{i_2}}(F^{e_{i_1}+e_{i_2}}_* \ba^{\lceil t(p^{e_{i_1} + e_{i_2}} - 1) \rceil} (c') ) ) \\
\subseteq \ldots \\
\subseteq \phi_{{i_1},\ldots, {i_n}}(F^{e_{i_1}+ \dots + e_{i_n}}_* \ba^{\lceil t(p^{e_{i_1} + \dots + e_{i_n}} - 1) \rceil} (c') ) ).
\end{split}
\]
Now apply Lemma \ref{LemConditionForTestElement} and set $c = c'^2$.
\end{proof}

We are now in a position to prove Theorem \ref{ThmMain} (that is, prove (iii) in the outline).

\begin{proof}[Proof of Theorem \ref{ThmMain}]
The containment $\supseteq$ is trivial, so we will prove the other containment.
Use Proposition \ref{PropCanReplaceByFiniteGenSet} to find finitely many \reduced{} homogeneous $\phi_i \in \sT$ generating a subalgebra $\sT'$ such that $\tau_b(R; \sT, \ba^t) = \tau_b(R; \sT', \ba^t)$.  By Proposition \ref{PropUniformExistenceOfTestElements}, we can choose $c \in \ba \cap R^{\circ}$ that is a \homgenbs{} test element for $(R, \sT^{i_1, \ldots, i_n}, \ba^t)$ for each subalgebra $\sT^{i_1, \ldots, i_n} = R\langle\phi_{i_1, \ldots, i_n}\rangle$ of $\sT$.  Again we use $\phi_{i_1, \ldots, i_n}$ to denote the product $\phi_{i_1} \cdot \dots \cdot \phi_{i_n}$

Now,
\[
\begin{split}
\tau_b(R; \sT', \ba^t) = \sum_{e \geq 0} \sum_{\phi \in \sT'_e} \phi(F^e_* (c) \ba^{\lceil t(p^e - 1) \rceil}) \\
= \sum_{e \geq 0} \left( \sum_{\phi_{i_1, \ldots, i_n} \in \sT'_e} \phi_{i_1, \ldots, i_n}(F^e_* (c) \ba^{\lceil t(p^e - 1) \rceil}) \right).
\end{split}
\]
Therefore we can choose generators for $\tau_b(R; \sT, \ba^t) = \tau_b(R; \sT', \ba^t)$ that are elements of $\phi_{i_1, \ldots, i_n}(F^e_* (c) \ba^{\lceil t(p^e - 1) \rceil})$ for various $\phi_{i_1, \ldots, i_n}$.  But any such generator is also contained clearly in
\[
\tau_b(R; \sS^{i_1, \ldots, i_n}, \ba^t)  = \tau_b(R; R\langle \phi_{i_1, \ldots i_n} \rangle, \ba^t)
\]
since $c$ is also a \homgenbs{} test element for $(R; \sS^{i_1, \ldots, i_n}, \ba^t)$.  This completes the proof.
\end{proof}

\section{Further remarks and questions}

\begin{remark}[Big verses finitistic test ideals]
The reader may have noticed that, in this paper, we dealt exclusively with the \emph{big} (aka non-finitistic) test ideal $\tau_b(R; \ba^t)$ and not the usual (aka finitistic) test ideal $\tau(R; \ba^t)$.  Roughly speaking, the big test ideal is made up of elements which ``test'' tight closure for all submodules of all modules, whereas the usual test ideal is made up of elements which ``test'' tight closure for all submodules of \emph{finitely generated} modules, see \cite{HochsterFoundations} and \cite{SchwedeCentersOfFPurity}.  The big test ideal and the usual test ideal are known to agree in many situations, see for example \cite{LyubeznikSmithStrongWeakFregularityEquivalentforGraded}, \cite{LyubeznikSmithCommutationOfTestIdealWithLocalization}, \cite{AberbachMacCrimmonSomeResultsOnTestElements}, \cite{HaraYoshidaGeneralizationOfTightClosure} and are conjectured to coincide in general (in particular, one always has the containment $\tau_b(R; \ba^t) \subseteq \tau(R; \ba^t)$).  Furthermore, if $K_X + \Delta$ is $\bQ$-Cartier, then $\tau_b(R; \Delta, \ba^t) = \tau(R; \Delta, \ba^t)$, see \cite[Theorem 2.8(2)]{TakagiInterpretationOfMultiplierIdeals}, \cite[Definition--Theorem 6.5]{HaraYoshidaGeneralizationOfTightClosure}.

The big test ideal is known to be much better behaved than the finitistic test ideal in general (for example, its formation is known to commute with localization and completion).  Finally, it is believed by experts that if it is discovered that $\tau_b(R; \ba^t) \subsetneq \tau(R; \ba^t)$ for some example, then the big test ideal is the ``correct'' notion in general.

Of course, it follows from this paper that if the finitistic test ideal $\tau(R; \ba^t)$ is equal to the same sum
\[
\sum_{K_X + \Delta \text{ $\bQ$-Cartier}} \tau_b(R; \Delta, \ba^t),
\]
then $\tau_b(R; \ba^t) = \tau(R; \ba^t)$.
\end{remark}

\begin{remark}[Relation with de Fernex and Hacon's multiplier ideal]
Suppose that $X = \Spec R$ is a normal variety in characteristic zero but $X$ is not necessarily $\bQ$-Gorenstein.  It has been asked whether $\mJ(X, \ba^t)$ (the multiplier ideal of de Fernex and Hacon) agrees with the (big) test ideal $\tau_b(R_p; \ba_p^t)$ after reduction to characteristic $p \gg 0$.  Initially, one might hope that the results of this paper might imply this result.  However, I believe that this paper only provides (strong) evidence that this is indeed the case.  The problem is that the $\Delta_i$ constructed in Corollary \ref{CorMain} seem to rely heavily on the particular characteristic we are working in, and so are probably not reduced from characteristic zero as well (at least not in a way in which their properties can be controlled).
\end{remark}

Consider the following question.

\begin{question}
Suppose that $(X, \ba^t)$ is a pair in characteristic $p > 0$.  Then does there exist a single effective $\bQ$-divisor $\Delta$ such that $K_X + \Delta$ is $\bQ$-Cartier with index not divisible by $p > 0$ and so that $\tau_b(X; \Delta, \ba^t) = \tau(X; \ba^t)$.
\end{question}

The work of de Fernex and Hacon suggests this may be true.  This question was also asked by the author and Karen Smith in \cite{SchwedeSmithLogFanoVsGloballyFRegular} and was affirmatively answered in the case that $X$ is strongly $F$-regular and $\ba = R$ (although the version where $\ba \neq R$ can also easily be obtained from the methods of this paper).

Finally, there are certain similarities between the methods employed in this paper and the methods of \cite[Theorem 4.3(i)]{SchwedeSmithLogFanoVsGloballyFRegular}.  The goal in \cite[Theorem 4.3(i)]{SchwedeSmithLogFanoVsGloballyFRegular} was also to find a map $\phi$ such that $(R, R\langle \phi \rangle)$ was strongly $F$-regular (although we phrased things in terms of finding the divisor associated to $\phi$, instead of $R\langle \phi \rangle$, as in \cite{SchwedeFAdjunction}).   One consideration that makes the setting of \cite{SchwedeSmithLogFanoVsGloballyFRegular} easier to work in, is that one must only find a single $\phi$ such that $1 \in \tau_b(R; R\langle \phi \rangle)$.  This $\phi$ is obtained explicitly by composing several potential such $\phi$ (but again, viewing them as divisors instead of maps).  The proof is thus somewhat more geometric.  Furthermore, much of the work we have to do in terms of keeping track of various test elements is unnecessary in the setting of \cite{SchwedeSmithLogFanoVsGloballyFRegular}.


\begin{thebibliography}{Tak04b}

\bibitem[AM99]{AberbachMacCrimmonSomeResultsOnTestElements}
{\sc I.~M. Aberbach and B.~MacCrimmon}: \emph{Some results on test elements},
  Proc. Edinburgh Math. Soc. (2) \textbf{42} (1999), no.~3, 541--549.
  {\sf\scriptsize MR1721770 (2000i:13005)}

\bibitem[Bli04]{BlickleMultiplierIdealsOnToric}
{\sc M.~Blickle}: \emph{Multiplier ideals and modules on toric varieties},
  Math. Z. \textbf{248} (2004), no.~1, 113--121. {\sf\scriptsize MR2092724
  (2006a:14082)}

\bibitem[BMS08]{BlickleMustataSmithDiscretenessAndRationalityOfFThresholds}
{\sc M.~Blickle, M.~Musta{\c{t}}{\u{a}}, and K.~Smith}: \emph{Discreteness and
  rationality of {F}-thresholds}, Michigan Math. J. \textbf{57} (2008), 43--61.

\bibitem[DH09]{DeFernexHacon}
{\sc T.~{De Fernex} and C.~Hacon}: \emph{Singularities on normal varieties},
  Compos. Math. \textbf{145} (2009), no.~2, 393--414.

\bibitem[Fed83]{FedderFPureRat}
{\sc R.~Fedder}: \emph{{$F$}-purity and rational singularity}, Trans. Amer.
  Math. Soc. \textbf{278} (1983), no.~2, 461--480. {\sf\scriptsize MR701505
  (84h:13031)}

\bibitem[Har01]{HaraInterpretation}
{\sc N.~Hara}: \emph{Geometric interpretation of tight closure and test
  ideals}, Trans. Amer. Math. Soc. \textbf{353} (2001), no.~5, 1885--1906
  (electronic). {\sf\scriptsize MR1813597 (2001m:13009)}

\bibitem[Har05]{HaraACharacteristicPAnalogOfMultiplierIdealsAndApplications}
{\sc N.~Hara}: \emph{A characteristic {$p$} analog of multiplier ideals and
  applications}, Comm. Algebra \textbf{33} (2005), no.~10, 3375--3388.
  {\sf\scriptsize MR2175438 (2006f:13006)}

\bibitem[HT04]{HaraTakagiOnAGeneralizationOfTestIdeals}
{\sc N.~Hara and S.~Takagi}: \emph{On a generalization of test ideals}, Nagoya
  Math. J. \textbf{175} (2004), 59--74. {\sf\scriptsize MR2085311
  (2005g:13009)}

\bibitem[HW02]{HaraWatanabeFRegFPure}
{\sc N.~Hara and K.-I. Watanabe}: \emph{F-regular and {F}-pure rings vs. log
  terminal and log canonical singularities}, J. Algebraic Geom. \textbf{11}
  (2002), no.~2, 363--392. {\sf\scriptsize MR1874118 (2002k:13009)}

\bibitem[HY03]{HaraYoshidaGeneralizationOfTightClosure}
{\sc N.~Hara and K.-I. Yoshida}: \emph{A generalization of tight closure and
  multiplier ideals}, Trans. Amer. Math. Soc. \textbf{355} (2003), no.~8,
  3143--3174 (electronic). {\sf\scriptsize MR1974679 (2004i:13003)}

\bibitem[Hoc07]{HochsterFoundations}
{\sc M.~Hochster}: \emph{Foundations of tight closure theory}, Lecture notes
  from a course taught on the University of Michigan Fall 2007 (2007).

\bibitem[HH90]{HochsterHunekeTC1}
{\sc M.~Hochster and C.~Huneke}: \emph{Tight closure, invariant theory, and the
  {B}rian\c con-{S}koda theorem}, J. Amer. Math. Soc. \textbf{3} (1990), no.~1,
  31--116. {\sf\scriptsize MR1017784 (91g:13010)}

\bibitem[HH94]{HochsterHunekeFRegularityTestElementsBaseChange}
{\sc M.~Hochster and C.~Huneke}: \emph{{$F$}-regularity, test elements, and
  smooth base change}, Trans. Amer. Math. Soc. \textbf{346} (1994), no.~1,
  1--62. {\sf\scriptsize MR1273534 (95d:13007)}

\bibitem[HR74]{HochsterRobertsRingsOfInvariants}
{\sc M.~Hochster and J.~L. Roberts}: \emph{Rings of invariants of reductive
  groups acting on regular rings are {C}ohen-{M}acaulay}, Advances in Math.
  \textbf{13} (1974), 115--175. {\sf\scriptsize MR0347810 (50 \#311)}

\bibitem[Kun76]{KunzOnNoetherianRingsOfCharP}
{\sc E.~Kunz}: \emph{On {N}oetherian rings of characteristic {$p$}}, Amer. J.
  Math. \textbf{98} (1976), no.~4, 999--1013. {\sf\scriptsize MR0432625 (55
  \#5612)}

\bibitem[Laz04]{LazarsfeldPositivity2}
{\sc R.~Lazarsfeld}: \emph{Positivity in algebraic geometry. {II}}, Ergebnisse
  der Mathematik und ihrer Grenzgebiete. 3. Folge. A Series of Modern Surveys
  in Mathematics [Results in Mathematics and Related Areas. 3rd Series. A
  Series of Modern Surveys in Mathematics], vol.~49, Springer-Verlag, Berlin,
  2004, Positivity for vector bundles, and multiplier ideals. {\sf\scriptsize
  MR2095472 (2005k:14001b)}

\bibitem[LLS08]{LazarsfeldLeeSmithSyzygiesOnSingular}
{\sc R.~Lazarsfeld, K.~Lee, and K.~E. Smith}: \emph{Syzygies of multiplier
  ideals on singular varieties}, Michigan Math. J. \textbf{57} (2008),
  511--521, Special volume in honor of Melvin Hochster. {\sf\scriptsize
  MR2492466}

\bibitem[LS99]{LyubeznikSmithStrongWeakFregularityEquivalentforGraded}
{\sc G.~Lyubeznik and K.~E. Smith}: \emph{Strong and weak {$F$}-regularity are
  equivalent for graded rings}, Amer. J. Math. \textbf{121} (1999), no.~6,
  1279--1290. {\sf\scriptsize MR1719806 (2000m:13006)}

\bibitem[LS01]{LyubeznikSmithCommutationOfTestIdealWithLocalization}
{\sc G.~Lyubeznik and K.~E. Smith}: \emph{On the commutation of the test ideal
  with localization and completion}, Trans. Amer. Math. Soc. \textbf{353}
  (2001), no.~8, 3149--3180 (electronic). {\sf\scriptsize MR1828602
  (2002f:13010)}

\bibitem[MR85]{MehtaRamanathanFrobeniusSplittingAndCohomologyVanishing}
{\sc V.~B. Mehta and A.~Ramanathan}: \emph{Frobenius splitting and cohomology
  vanishing for {S}chubert varieties}, Ann. of Math. (2) \textbf{122} (1985),
  no.~1, 27--40. {\sf\scriptsize MR799251 (86k:14038)}

\bibitem[SS09]{SchwedeSmithLogFanoVsGloballyFRegular}
{\sc K.~Schwede and K.~Smith}: \emph{Globally {$F$}-regular and log {F}ano
  varieties}, arXiv:0905.0404, to appear in Advances in Mathematics.

\bibitem[Sch08a]{SchwedeCentersOfFPurity}
{\sc K.~Schwede}: \emph{Centers of {$F$}-purity}, arXiv:0807.1654, to appear in
  Mathematische Zeitschrift.

\bibitem[Sch08b]{SchwedeSharpTestElements}
{\sc K.~Schwede}: \emph{Generalized test ideals, sharp {$F$}-purity, and sharp
  test elements}, Math. Res. Lett. \textbf{15} (2008), no.~6, 1251--1261.
  {\sf\scriptsize MR2470398}

\bibitem[Sch09]{SchwedeFAdjunction}
{\sc K.~Schwede}: \emph{{$F$}-adjunction}, Algebra Number Theory \textbf{3}
  (2009), no.~8, 907--950.

\bibitem[Smi00]{SmithMultiplierTestIdeals}
{\sc K.~E. Smith}: \emph{The multiplier ideal is a universal test ideal}, Comm.
  Algebra \textbf{28} (2000), no.~12, 5915--5929, Special issue in honor of
  Robin Hartshorne. {\sf\scriptsize MR1808611 (2002d:13008)}

\bibitem[Tak04a]{TakagiInversion}
{\sc S.~Takagi}: \emph{F-singularities of pairs and inversion of adjunction of
  arbitrary codimension}, Invent. Math. \textbf{157} (2004), no.~1, 123--146.
  {\sf\scriptsize MR2135186}

\bibitem[Tak04b]{TakagiInterpretationOfMultiplierIdeals}
{\sc S.~Takagi}: \emph{An interpretation of multiplier ideals via tight
  closure}, J. Algebraic Geom. \textbf{13} (2004), no.~2, 393--415.
  {\sf\scriptsize MR2047704 (2005c:13002)}

\bibitem[Tak08]{TakagiPLTAdjoint}
{\sc S.~Takagi}: \emph{A characteristic {$p$} analogue of plt singularities and
  adjoint ideals}, Math. Z. \textbf{259} (2008), no.~2, 321--341.
  {\sf\scriptsize MR2390084 (2009b:13004)}

\end{thebibliography}

\providecommand{\bysame}{\leavevmode\hbox to3em{\hrulefill}\thinspace}
\providecommand{\MR}{\relax\ifhmode\unskip\space\fi MR}
\providecommand{\MRhref}[2]{%
  \href{http://www.ams.org/mathscinet-getitem?mr=#1}{#2}
}
\providecommand{\href}[2]{#2}

\end{document}